\documentclass[12pt,reqno]{amsart}
\usepackage{amsmath}
\usepackage{amssymb}
\usepackage[]{geometry}
\usepackage{epsfig}
\usepackage[normalem]{ulem}
\usepackage[usenames,dvipsnames]{color}
\usepackage{todonotes}

\usepackage{enumerate}

\usepackage{cite}

\numberwithin{equation}{section}

\makeatletter
\newtheorem*{rep@theorem}{\rep@title}
\newcommand{\newreptheorem}[2]{%
\newenvironment{rep#1}[1]{%
 \def\rep@title{#2 \ref{##1}}%
 \begin{rep@theorem}}%
 {\end{rep@theorem}}}
\makeatother

\newtheorem{theorem}{Theorem}[section]
\newtheorem{lemma}[theorem]{Lemma}

\newtheorem{corollary}[theorem]{Corollary}

\newtheorem{problem}[theorem]{Problem}

\newreptheorem{theorem}{Theorem}

\newcommand{\N}{{\mathbb N}}

\newcommand{\E}{\mathbb E}

\newcommand{\Prob}{\mathbb{P}}

\newcommand{\varx}{x}
\newcommand{\vary}{y}

\begin{document}

\makeatletter
\providecommand\@dotsep{5}
\renewcommand{\listoftodos}[1][\@todonotes@todolistname]{%
  \@starttoc{tdo}{#1}}
\makeatother

\title[Subgraph Games in the Semi-Random Graph Process]{Subgraph Games in the Semi-Random Graph Process and Its Generalization to Hypergraphs}

\author[N.\ Behague]{Natalie Behague}
\address{Mathematics Institute, University of Warwick, Coventry, UK}
\email{\texttt{natalie.behague@warwick.ac.uk}}

\author[T.G.\ Marbach]{Trent G.\ Marbach}
\address{Department of Mathematics, Toronto Metropolitan University, Toronto, ON, Canada}
\email{\tt trent.marbach@gmail.com}

\author[P.\ Pra\l{}at]{Pawe\l{} Pra\l{}at}
\address{Department of Mathematics, Toronto Metropolitan University, Toronto, ON, Canada}
\email{\texttt{pralat@torontomu.ca}}
\thanks{Pawe\l{} Pra\l{}at was partially supported by the Natural Sciences and Engineering Research Council of Canada}

\author[A.\ Ruci\'nski]{Andrzej Ruci\'nski}
\address{Faculty of Mathematics and Computer Science, Adam Mickiewicz University, Pozna\'n, Poland}
\email{\texttt{rucinski@amu.edu.pl}}
\thanks{Andrzej Ruci\'nski was supported by Narodowe Centrum Nauki, grant 2018/29/B/ST1/00426}


\maketitle

\begin{abstract}
The semi-random graph process is a single-player game that begins with an empty graph on $n$ vertices. In each round, a vertex $u$ is presented to the player independently and uniformly at random. The player then responds by selecting a vertex $v$ and adds the edge $uv$ to the graph. For a fixed (monotone) increasing graph property, the player's objective is to force the graph to satisfy this property with high probability in as few rounds as possible.

We focus on the problem of constructing a subgraph isomorphic to an arbitrary, fixed graph $H$.  
In~\cite{beneliezer2019semirandom}, it was proved that asymptotically almost surely one can construct $H$ in $t$ rounds, for any $t\gg n^{(d-1)/d}$ where $d \ge 2$ is the degeneracy of~$H$. It was also proved that this result is sharp for $H = K_{d+1}$ 
 and conjectured that it is so for all graphs $H$. We prove this conjecture, and the conjecture's generalization to a semi-random $s$-uniform hypergraph process for every $s\ge2$.

\end{abstract}

\section{Introduction\label{intro}}

In this paper, we consider the \emph{semi-random process} suggested by Peleg Michaeli (see~\cite{beneliezer2020fast} and~\cite[Acknowledgements]{beneliezer2019semirandom}), formally introduced in~\cite{beneliezer2019semirandom}, and studied recently in \cite{beneliezer2020fast,gao2022fully,frieze2022hamilton,gao2020hamilton,macrury2022sharp,ham_cycles_preprint,gamarnik2023cliques,Harjas}. This process can be viewed as a ``one player game''.

\subsection{Definitions}\label{definitions}

The process starts from $G_0$, the empty graph on the vertex set $[n]:=\{1,2,\ldots,n\}$ where $n \ge 1$. In each step $t\geq 1$, a vertex $u_t$ is chosen uniformly at random from $[n]$. Then, the player (who is aware of graph $G_{t-1}$ and vertex $u_t$) must select a vertex $v_t$ and add the edge $u_tv_t$ to $G_{t-1}$ to form $G_{t}$. The player aims to build a (multi)graph satisfying a given property $\mathcal{P}$ as quickly as possible.
It is convenient to think of $u_t$ as receiving a {\bf square}, and $v_t$ as receiving a {\bf circle}, so every edge in $G_t$ joins a square with a circle. Equivalently, we may view $G_t$ as a directed graph where arcs go from $u_i$ to $v_i$, $i=1,\dots,t$.
To make the process well defined, we allow parallel edges (for example, if some vertex receives $n$ squares, a parallel edge is necessary).

A \emph{strategy} $\mathcal{S}$ is defined by specifying, for each $n \ge 1$, a sequence of functions $(f_{t})_{t=1}^{\infty}$, where for each $t \in \N$, $f_t(u_1,v_1,\ldots, u_{t-1},v_{t-1},u_t)$ is a distribution over $[n]$ that depends on the vertex $u_t$, and the history of the process up until step $t-1$. Then, $v_t$ is chosen according to this distribution. If $f_t$ is an atomic distribution, that is \emph{non-random}, then $v_t$ is fully determined by $u_1,v_1, \ldots ,u_{t-1},v_{t-1},u_t$.
We denote by $(G_{i}(n)[\mathcal{S}])_{i=0}^{t}$  the sequence of random (multi)graphs obtained by following the strategy $\mathcal{S}$ for $t$ rounds; we shorten $G_{t}(n)[\mathcal{S}]$ to $G_t$ or $G_{t}(n)$ when clear.

Suppose $\mathcal{P}$ is a monotonically increasing property of graphs. We say that a function $\tau_{\mathcal{P}}(n)$ is \emph{a threshold for $\mathcal{P}$} if the following two conditions hold:
\begin{itemize}
\item[(a)] there exists a strategy $\mathcal{S}$ so that if
$t:=t(n)\gg \tau_{\mathcal{P}}(n)$, then
\newline
$
{\lim_{n \to \infty} \Prob[ G_t \in \mathcal{P} ] = 1,\;}
$
and
\item[(b)] for every strategy~$\mathcal{S}$, if $t:=t(n)=o( \tau_{\mathcal{P}}(n))$, then
$
\lim_{n \to \infty} \Prob[ G_t \in \mathcal{P} ] = 0.
$
\end{itemize}
\noindent Here and throughout we write $a_n\gg b_n$ if $b_n=o(a_n)$.
We also say that an event holds \emph{asymptotically almost surely} (\emph{a.a.s.}) if it holds with probability tending to one as $n \to \infty$.
\subsection{Main Result}\label{main result}

In this paper, we focus on the problem of constructing a sub-graph isomorphic to an arbitrary, fixed graph $H$.
Let $\mathcal{P}_H$ be the property that $H \subseteq G_t$.
It turns out that the threshold $\tau_{\mathcal{P}_H}$ can be determined in terms of the degeneracy of $H$.

For a given $d \in \N$, a graph $H$ is  \emph{$d$-degenerate} if every sub-graph $H'\subseteq H$ has minimum degree $\delta(H')\le d$. The \emph{degeneracy} of $H$ is the smallest value of $d$ for which $H$ is $d$-degenerate.
It was proved in~\cite{beneliezer2019semirandom} that for any graph $H$  of degeneracy $d \in \N$, $\tau_{\mathcal{P}_H} \le n^{(d-1)/d}$.
\begin{theorem}[\hspace{1sp}\protect{\cite[Theorem~1.10]{beneliezer2019semirandom}}]\label{thm:upper_bound}
Let $H$ be a fixed graph of degeneracy $d \in \N$. Then, there exists a strategy $\mathcal{S}$ so that whenever $t\gg n^{(d-1)/d}$,
$$
\lim_{n \to \infty} \Prob[ G_t \in \mathcal{P}_H ] = 1.
$$
\end{theorem}
\noindent For completeness and as a warm-up, we re-prove this theorem in Section~\ref{upper}.

Note that for $d=1$, that is, when $H$ is a forest, Theorem \ref{thm:upper_bound} implies immediately that $\tau_{\mathcal{P}_H} = 1$.
For $d\ge2$, it was proved in~\cite{beneliezer2019semirandom} that $\tau_{\mathcal{P}_H} = n^{(d-1)/d}$  when $H=K_{d+1}$, the complete graph on $d+1$ vertices, and conjectured that the equality holds for all graphs of degeneracy $d$. As our main result, we prove this conjecture here.

\begin{theorem}\label{thm:lower_bound}
Let $H$ be a fixed graph of degeneracy $d \ge 2$. Then, for any strategy~$\mathcal{S}$, if  $t =o( n^{(d-1)/d} )$, then
$$
\lim_{n \to \infty} \Prob[ G_t \in \mathcal{P}_H ] = 0.
$$
\end{theorem}

Combining Theorems~\ref{thm:upper_bound} and~\ref{thm:lower_bound} we get the following corollary.
\begin{corollary}
Let $H$ be a fixed graph of degeneracy $d \in \N$. Then, $\tau_{\mathcal{P}_H} = n^{(d-1)/d}$.
\end{corollary}

\subsection{Background}\label{background}

The semi-random process was also recently studied in the context of perfect matchings~\cite{gao2022perfect} and Hamilton cycles~\cite{gao2020hamilton,gao2022fully,frieze2022hamilton,ham_cycles_preprint}. For both structures, since the goal is to create a spanning subgraph with bounded maximum degree, the length of the process leading to constructing them must be, trivially, of order $\Omega(n)$. In all above papers, a matching bound of $O(n)$ is established, however, the multiplicative constants were not determined precisely. Cliques, chromatic number, and independent sets were considered in~\cite{gamarnik2023cliques}.

Perfect matchings and Hamilton cycles are just two special cases of the property of containing a given graph $H_n$ as a spanning subgraph.
As was reported in~\cite{beneliezer2020fast}, Noga Alon asked, more generally, whether for any fixed sequence of graphs $H_n$ with maximum degree $\Delta(H_n)\le\Delta$ for all $n$ and $H_n$ containing at most $n$ vertices, one can construct a copy of $H_n$ in $G_t$ on $n$ vertices a.a.s.\ for $t=O(n)$.  This question was answered positively in a strong sense in~\cite{beneliezer2020fast}, where it was shown that such an $H_n$ can be constructed a.a.s.\ in $(3\Delta/2+o(\Delta))n$ rounds. They also proved that if $\Delta \gg \log(n)$, then this upper bound improves to $(\Delta/2+o(\Delta))n$ rounds. Note that this result applies to fixed subgraphs too, but this bound is far too weak.
Indeed, we will show that the property of containing a fixed subgraph has a threshold of order $o(n)$. Consequently, we will be interested in finding the correct exponent of $n$ rather than multiplicative constants.

The semi-random process may be extended or generalized in various ways. For example, in~\cite{gilboa2021semi} the authors consider a no-replacement variant of the process in which squares follow a permutation of vertices selected uniformly at random. Once each vertex is covered with a square, another random permutation is drawn, and the process continues. Another variant was studied in~\cite{burova2022semi} in which a random spanning tree of $K_n$ is presented to the player who can keep one of the edges. In~\cite{Harjas}, the process presents $k$ squares, and to create an edge the player selects one of them, and freely chooses a circle to connect to.

\subsection{Hypergraphs}\label{main_results_hypergraphs}

In this paper, we propose a natural generalization of the semi-random process to hypergraphs (cf.~\cite{molloy2023matchings}). Fix $r \ge 1$ to be the number of randomly selected vertices per step, and $s \ge r$ to be the uniformity of the hypergraph. The process starts from $G^{(r,s)}_0$, the empty hypergraph on the vertex set $[n]$, where $n \ge 1$. In each step $t\geq 1$, a set $U_t$ of $r$ vertices is chosen uniformly at random from $[n]$. Then, the player replies by selecting a set of $s-r$ vertices $V_t$ and the edge $U_t \cup V_t$ is added to $G^{(r,s)}_{t-1}$ to form $G^{(r,s)}_{t }$. We assume that $U_t$ and $V_t$ are disjoint so that the resulting hypergraph is an $s$-uniform hypergraph, or shortly an \emph{$s$-graph}. As was the case with graphs, in order for the process to be well defined
we will allow parallel edges.

If $r=1$ and $s=2$, then this is the semi-random graph process described above. On the other hand, if $r=s$ (that is, the player chooses $V_t=\emptyset$ for all $t$), then $G_t^{(r,r)}$ is just a uniform random $r$-graph process with $t$ edges selected with repetitions.

In this paper, we will focus on the case where $r=1$. In this case, at each step one vertex is randomly selected and the player chooses $s-1$ vertices. For simplicity, we will refer to the $s$-graph $G_t^{(1,s)}$ simply as  $G_t^{(s)}$.


As before, the goal of the player is to build an $s$-graph $G_t^{(s)}$ satisfying a given property $\mathcal{P}$ as quickly as possible, and we focus on the property $\mathcal{P}_H$ of possessing a sub-$s$-graph isomorphic to an arbitrary, fixed $s$-graph $H$.
We define strategies and the threshold $\tau_{\mathcal{P}_H}$ identically to the graph case.

In Section~\ref{sec:hypergraphs r=1} we show that for uniform hypergraphs, the case $r=1$ resembles the graph case and the degeneracy of an $s$-uniform hypergraph $H$ is still the only parameter that affects the threshold for the property $H \subseteq G_t^{(s)}$. As for graphs, for a given $d \in \N$, a hypergraph $H$ is  \emph{$d$-degenerate} if every sub-hypergraph $H'\subseteq H$ has minimum degree $\delta(H')\le d$ (where the minimum degree of a hypergraph is the minimum degree over all vertices).
The \emph{degeneracy} of $H$ is the smallest value of $d$ for which $H$ is $d$-degenerate.

In particular, we have the following theorems that are counterparts of Theorem~\ref{thm:upper_bound} and Theorem~\ref{thm:lower_bound}.

\begin{theorem}\label{thm:hypergraph_r=1_upper_bound}
Let $r=1$, $s \ge 2$, and let $H$ be a fixed $s$-uniform hypergraph of degeneracy $d \in \N$. Then, there exists a strategy $\mathcal{S}$ so that whenever $t \gg n^{(d-1)/d} $,
$$
\lim_{n \to \infty} \Prob[ G_t^{(s)} \in \mathcal{P}_H ] = 1.
$$
\end{theorem}

\begin{theorem}\label{thm:hypergraph_r=1_lower_bound}
Let $r=1$,  $s \ge 2$, and let $H$ be a fixed $s$-uniform hypergraph of degeneracy $d \ge 2$. Then, for any strategy $\mathcal{S}$, if $t = o( n^{(d-1)/d})$, then
$$
\lim_{n \to \infty} \Prob[ G_t^{(s)} \in \mathcal{P}_H ] = 0.
$$
\end{theorem}

 As a result, combining Theorems~\ref{thm:hypergraph_r=1_upper_bound} and~\ref{thm:hypergraph_r=1_lower_bound}, we get the following corollary.
\begin{theorem}\label{thm:any_hypergraph}
Let $r=1$ and let $H$ be a fixed $s$-uniform hypergraph of degeneracy $d \in \N$. Then, $\tau_{\mathcal{P}_H} = n^{(d-1)/d}$.
\end{theorem}
\noindent The proofs of these results follow the same approach as in the graph case.

\medskip

\subsection{Organization}
In the next section we prove Theorems \ref{thm:upper_bound} and \ref{thm:lower_bound}, while in Section \ref{sec:hypergraphs r=1} we concentrate on Theorems~\ref{thm:hypergraph_r=1_upper_bound} and~\ref{thm:hypergraph_r=1_lower_bound}. 
The last section presents a number of open problems, including the problem of the hypergraph case when the number of randomly selected vertices $r$ satisfies $1 < r < s$. Some further results on this case will be presented in a follow-up paper.

\section{Proofs for graphs} \label{sec:graphs}

\subsection{Outline} First, in Subsection \ref{upper}, we prove Theorem \ref{thm:upper_bound} which sets an upper bound on $\tau_{{\mathcal P}_H}$. Then, in Subsection \ref{lower}, a proof of Theorem \ref{thm:lower_bound} is given that provides a matching lower bound. A probabilistic lemma, established in Subsection \ref{uselem}, is utilized in both proofs. The proof of Theorem~\ref{thm:lower_bound} is much more involved and requires an auxiliary notion of vertex weighting $w$. In Lemma~\ref{lem:weight_>_min_deg} we show that the weights are bounded from below by $\delta(H)$. Then, in Lemma~\ref{lem:weight_related_to_H}, we bound the number of possible images of a vertex $v$ of $H$ in $G_t$ in terms of $w(v)$. Combined, these two lemmas yield the proof of Theorem~\ref{thm:lower_bound}. Before all that, we include a brief compendium on the degeneracy of graphs and hypergraphs.

\subsection{Degeneracy}\label{degeneracy}
We start with some useful basic facts about degeneracy.
Recall that for a given $d \in \N$, a hypergraph $H$ is  \emph{$d$-degenerate} if every sub-hypergraph $H'\subseteq H$ has minimum degree $\delta(H')\le d$.
 The \emph{degeneracy} of $H$ is the smallest value of $d$ for which $H$ is $d$-degenerate.

The \emph{$d$-core} of a hypergraph $H$ is the maximal (with respect to inclusion) induced subgraph $H'\subseteq H$ with minimum degree $\delta(H')\ge d$. (Note that the $d$-core is well defined, though it may be empty. Indeed, if $S \subseteq V(H)$ and $T \subseteq V(H)$ induce sub-hypergraphs with minimum degree at least $d$, then the same is true for $S \cup T$.) If $H$ has degeneracy $d$ then it has a non-empty $d$-core. Indeed, by definition, $H$ is \emph{not} $(d-1)$-degenerate and so it has a sub-hypergraph $H'$ with  $\delta(H')\ge d$.
We immediately get that if $H$ has degeneracy $d$, then there exists an ordering of the vertices of $H$, $(v_1, v_2, \ldots , v_k)$, such that for each $\ell \in [k]$ vertex $v_\ell$ has degree at most $d$ in the sub-hypergraph induced by the set $\{v_1, v_2, \ldots, v_{\ell}\}$.

 For graphs, this implies a useful reformulation of degeneracy: a graph $H$ is $d$-degenerate if and only if the edges of $H$ can be oriented to form a directed acyclic graph $D$ with maximum out-degree at most $d$. In other words, there exists a permutation of the vertices of $H$, $(v_1, v_2, \ldots, v_k)$, such that for every directed edge $(v_i,v_j)\in D$ we have  $i > j$ and the out-degrees are at most $d$.
For example, the degeneracy of the complete graph $K_k$ is $k-1$, and any acyclic tournament embodies the aforementioned orientation.

\subsection{Useful Lemma }\label{uselem}
Let us first state the following simple but useful lemma.
 The proofs of Theorems \ref{thm:upper_bound} and \ref{thm:lower_bound}, as well as (since the lemma does not depend on~$s$) those of Theorems \ref{thm:hypergraph_r=1_upper_bound} and \ref{thm:hypergraph_r=1_lower_bound}, will rely on it.

 \begin{lemma}\label{lem:pick_same_vertex}
Let $t = o(n)$ and let $\omega=\omega(n)$ be any function tending to infinity as $n \to \infty$. Let $\varx \in \N$ and let $X_t^{(\varx)}$ be the number of vertices in $G_t$ with precisely $\varx$ squares on them, that is, the number of vertices in $G_t$ with out-degree $\varx$. Then the following holds:
\begin{itemize}
    \item [(a)] $\E X_t^{(\varx)}  = (1-o(1)) \frac {t^\varx} {\varx! n^{\varx-1}}$.
    \item [(b)] If $t = n^{(\varx-1)/\varx} / \omega$, then a.a.s.\ $X_t^{(\varx)} = 0$.
    \item [(c)] If $t = n^{(\varx-1)/\varx} \omega$, then for any $\vary \in \N$ a.a.s.\ $X_t^{(\varx)} \ge \vary$.
\end{itemize}
\end{lemma}
\begin{proof} Let $Y_t(i)$, $i=1,\dots,n$, be the number of  squares  on vertex $i$ in $G_t$. It follows from our random model that each $Y_t(i)$ is a random variable with binomial distribution $Bin(t,1/n)$. Note, however, that the random variables $Y_t(i)$ for $i=1,\dots,n$ are \emph{not} independent. The same observation applies to the indicator random variables $I_t(i)^{(\varx)}$, where for each $i=1,\dots n$, $I_t(i)^{(\varx)}=1$ if $Y_t(i)=\varx$ and 0 otherwise.
Thus
$$X_t^{(\varx)}=\sum_{i=1}^nI_t(i)^{(\varx)},$$
and, since $t=o(n)$ and $\varx$ is a constant, we immediately get that
\begin{eqnarray*}
\E X_t^{(\varx)}  &=& n \binom{t}{\varx} \left( \frac{1}{n} \right)^\varx \left(1- \frac{1}{n} \right)^{t-\varx} \\
&=& (1-o(1)) \frac {nt^\varx}{\varx! n^\varx} \exp \left( - \frac {t-\varx}{n} + O( t/n^2 ) \right) \\
&=& (1-o(1)) \frac {t^\varx}{\varx! n^{\varx-1}}.
\end{eqnarray*}
If $t = n^{(\varx-1)/\varx} / \omega$, then $\E X_t^{(\varx)}  = o(1)$ and so a.a.s.\ $X_t^{(\varx)} = 0$ by the first moment method. On the other hand, if $t = n^{(\varx-1)/\varx} \omega$, then $\E X_t^{(\varx)}  \sim\omega^\varx/x!\to \infty$ as $n \to \infty$.
Set $X:= X_t^{(\varx)}$ for convenience. To turn the above estimate of expectation $\E X$ into the desired lower bound on $X$ itself, we are going to apply the second moment method, or Chebyshev's inequality, with the variance  expressed in terms of the second factorial moment (this form fits well the cases when all summands constituting $X$ are pairwise  dependent):
\begin{equation}\label{2ndMM}
\Prob\left(|X-\E X|>\frac12\E X\right)\le\frac{4\mathrm{Var}(X)}{(\E X)^2}=4\left(\frac{\E(X(X-1))}{(\E X)^2}+\frac1{\E X}-1\right).
\end{equation}
Since $\E X\to\infty$ as $n \to \infty$, it suffices to show that $\E(X(X-1))\sim(\E X)^2$. By symmetry,
$$\E(X(X-1))=n(n-1)\Prob(I_t(1)^{(\varx)}=I_t(2)^{(\varx)}=1),$$
while
$$\Prob(I_t(1)^{(\varx)}=I_t(2)^{(\varx)}=1)=\binom tx\binom{t-x}x\left(\frac1n\right)^{2x}\left(1-\frac2n\right)^{t-2x}\sim\frac{t^{2x}}{x!^2n^{2x}},$$
and, thus, indeed, $\E(X(X-1))\sim(\E X)^2$. Consequently, a.a.s.\ $X_t^{(\varx)} \ge\omega^\varx/(3x!)$ and, in particular, $X_t^{(\varx)} \ge \vary$, regardless of the value of $\vary \in \N$.
\end{proof}

\bigskip

\subsection{Upper Bound }\label{upper}

In this section, we will re-prove Theorem~\ref{thm:upper_bound}. We do it for completeness as well as to highlight challenges in proving the lower bound.


\begin{proof}[Proof of Theorem~\ref{thm:upper_bound}]
Let $H$ be a graph on $k$ vertices $V(H)=\{v_1, v_2, \ldots, v_k\}$ of degeneracy $d \in \N$. As such, we may assume that for each $\ell \in [k]$, vertex $v_\ell$ has at most $d$ neighbours among $\{v_1, v_2, \ldots, v_{\ell-1}\}$. We orient edges of $H$ so that for all edges $v_iv_j$ it holds that $j < i$. As a result, the maximum out-degree is equal to $d$.

The player can create the oriented graph $H$ in $t \gg n^{(d-1)/d}$ rounds by using the following simple strategy. The process is divided into $k$ phases labelled with $\ell \in [k]$, each consisting of $t/k$ rounds. We proceed by an inductive argument. At the beginning of phase $\ell$, we assume that a copy of the induced subgraph $H[\{v_1, v_2, \ldots, v_{\ell-1}\}]$ has been already created in $G_{(\ell -1)t/k}$.

Note that the property is vacuously satisfied at the beginning of phase $1$. At the beginning of phase $2$, we may select any vertex to obtain a copy of $H[v_1]$. Therefore, let $\ell \ge 2$. Let us fix one such copy and let $u_i$ be the image of $v_i$, $i=1,\dots,\ell-1$ in that copy.


Let $N_{\ell} \subseteq \{v_1, v_2, \ldots, v_{\ell-1}\}$ be the neighbours of $v_{\ell}$ in $H$ that come earlier in the vertex ordering. By construction, $h:=|N_{\ell}| \le d$.
Let $w_1,\dots,w_h$ be the images of the vertices of $N_\ell$ in the fixed copy of $H[\{v_1, v_2, \ldots, v_{\ell-1}\}]$ in $G_{(\ell-1)t/k}$.

The goal of the player (in this phase) is to create an image $u_\ell$ of vertex $v_{\ell}$ that is adjacent to $w_1,\dots,w_h$. In order to achieve this, when some vertex receives its $i$th square during this phase, $1\le i\le h$, the player simply connects this vertex with $w_i$.
It follows from Lemma~\ref{lem:pick_same_vertex}(c) with $\varx=d$ and $\vary=k$ that a.a.s.\ at least $k$ vertices receive $d$ squares during this phase, in which case we can find such a vertex distinct from $u_1,\dots,u_{\ell-1}$. Therefore a.a.s.\ the fixed copy of $H[\{v_1, v_2, \ldots, v_{\ell-1}\}]$ can be extended to a copy of $H[\{v_1, v_2, \ldots, v_{\ell-1},v_\ell\}]$.
Since the number of phases is $k=O(1)$, a.a.s.\ a copy of $H$ is created in $k$ phases, and the proof is finished.
\end{proof}

\subsection{Lower Bound}\label{lower}

In this section, we prove the main result of this paper, Theorem~\ref{thm:lower_bound}. 

Let $H$ be a graph on $k$ vertices and $m$ edges that may contain loops and parallel edges. Fix an ordering of the edges $(e_1,e_2,\ldots,e_m)$ of $H$, and fix an orientation of each edge. We will analyze the probability of an oriented copy of $H$ arising in $G_t$ where the edges of $H$ are added to $G_t$ in this specified order and the edge orientations in $G_t$ (from squares to circles) respect the edge orientations of $H$. Since, for a fixed $H$, there is only a finite number of ways to order and orient the edges, we can sum these probabilities to get the desired bound on the occurrence of \emph{any} copy of $H$ in $G_t$. Later on, we will formally prove this simple observation. For now, let us restrict ourselves to a given orientation and a given order of edges, and assume that the player's goal is to create a copy of $H$ with these additional constraints.

\medskip

To highlight the main challenge in proving the result, consider a simple example with $H=C_4$, the cycle of length 4. If the cycle is oriented so that one of the vertices has out-degree 2, then it follows immediately from Lemma~\ref{lem:pick_same_vertex}(b) that a.a.s.\ one cannot accomplish the task in $o\left(\sqrt{n}\right)$ rounds, and we are done. However, if the cycle is oriented so that each vertex has out-degree 1, then no non-trivial bound can be deduced from the lemma.

 In order to deal with all possible scenarios, for a given orientation and order, we define a weight function on the vertices of the graph $H$. This function is meant to measure how much of the difficulty in creating a copy of $H$ hinges upon a given vertex.
   We will then show (see Lemma \ref{lem:weight_>_min_deg}) that if $H$ is $d$-degenerate, then \emph{all} vertices of its $d$-core $H'$ have weight at least $d$. On the other hand, we will show (see Lemma \ref{lem:weight_related_to_H}) that even if $H'$ contains just one vertex of weight at least $d$, then the expected number of copies of $H'$ in $G_t$ is $O( t^d / n^{d-1} )$, regardless of the strategy of the player. As a result, if $t=o(n^{(d-1)/d})$, then the expectation tends to zero, and the desired conclusion holds by the first moment method: a.a.s.\ there is no copy of $H'$, and thus of $H$, in $G_t$.

\medskip

In order to prove Lemma \ref{lem:weight_>_min_deg} it is helpful to allow directed graphs that contain loops, including potentially several loops on the same vertex, to make the inductive step work. We call such a graph \emph{loopy}.

As promised, we recursively define a weight function $w_H:V \rightarrow \N\cup\{0\}$ on the vertices of a loopy graph $H =(V,E)$ that is dependent on the edge order and orientation. Let $H_0$ be the edgeless graph on vertex set $V$ and define the weighting $w_{H_0}:V \rightarrow \N\cup\{0\}$ to be uniformly zero. For $1 \le i \le m = |E|$, let $H_i$ have vertex set $V$ and edge set $\{ e_1, e_2, \ldots,e_i \}$ (so $H_m = H$). In particular, $H_i$ is $H_{i-1}$ with edge $e_i$ added. Let $e_i$ be the directed edge $x_i \rightarrow y_i$ (where we may have $y_i = x_i$). Define $w_{H_i}:V \rightarrow \N\cup\{0\}$ by
$$
    w_{H_i}(x_i) = w_{H_{i-1}}(x_i) + 1
$$
and for all other vertices $v \in V$,
$$
    w_{H_i}(v) = \begin{cases} \max\{w_{H_i}(x_i), w_{H_{i-1}}(v)\} &\text{if $x_i \leadsto v$ in $H_i$} \\
     w_{H_{i-1}}(v) &\text{otherwise,}
     \end{cases}
$$
where $x_i \leadsto v$ denotes that there is a directed path from $x_i$ to $v$. See Figure~\ref{fig:weighting} for an example of the updating rule.
Note that for every $i$ the weights of vertices on a directed path in $H_i$ form a non-decreasing sequence and that for every $v\in V$ we have $w_i(v)\le i$.

\begin{figure}[ht]
    \centering
    \includegraphics[scale=1.1]{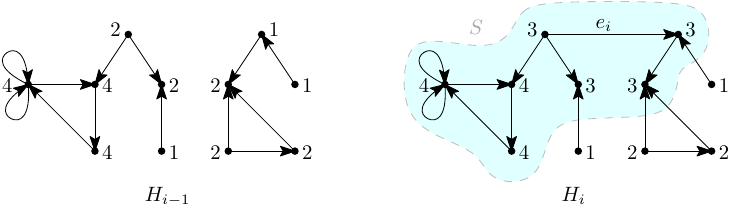}
    \caption{An example of a vertex weighting on $H_{i-1}$ and the updated weighting on $H_i$.}
    \label{fig:weighting}
\end{figure}

The weight of a vertex $v$ relates implicitly to the number of vertices that are images of $v$ in the copies of $H$ in $G_t$, where a higher weight means fewer copies (cf. Lemma \ref{lem:weight_related_to_H} below).
In particular, it counts how many times the random process must pick $v$ in order to create a copy of $H$ with an additional technical constraint that the weights cannot decrease along directed paths.
The intuition is as follows. Suppose there is a number of images of $x_i$ lying within copies of $H_{i-1}$ in $G_{t}$. Only a fraction of them will become an image of $x_i$ in a copy of $H_i$, as the random process must choose them (assign a square)  at a later time.
Similarly, as the player assigns only one circle at a time, the pool of images of $y_i$ in copies of $H_i$ will shrink as the process progresses. Hereditarily,  the same applies to the vertices further away from $x_i$ in $H_i$ but accessible from it by directed paths.

\begin{lemma}\label{lem:weight_>_min_deg}
If a loopy graph $H$ has minimum degree $\delta = \delta(H)$ (where a loop edge $v \rightarrow v$ contributes one to the degree of $v$), then $w_H(v) \ge \delta$ for every vertex $v$.
\end{lemma}
\begin{proof}
We prove the statement by induction on the number of edges $m = |E(H)|$. The base case is trivial. If $m = 0$, then $H= H_0$, $\delta(H_0)=0$, and all vertices have weight zero by definition.

For the inductive step, assume that $m \ge 1$ and the result holds for all graphs with fewer than $m$ edges. Let $\delta=\delta(H_m)$. Clearly, $H_{m-1}$ has minimum degree $\delta$ or $\delta - 1$. If $H_{m-1}$ has minimum degree $\delta$ then we are done, as $w_{H_m}(v) \ge w_{H_{m-1}}(v) \ge \delta$ for every vertex $v$.

Suppose then that $H_{m-1}$ has minimum degree $\delta - 1$ and that $w_{H_{m-1}}(v) \ge \delta - 1$ for all $v$. Let $e_m = x_m \rightarrow y_m$. We have $w_{H_m}(x_m) = w_{H_{m-1}}(x_m) + 1 \ge \delta$. Let $S$ be the set of all vertices $v$ with a directed path from $x_m$ to $v$ in $H_m$.
We know that $x_m \in S$ (there is a degenerate directed path from $x_m$ to $x_m$) and $w_{H_m}(v) \ge w_{H_m}(x_m) \ge \delta$ for any $v \in S$. If $S = V(H_m)$ then we are done, so suppose not and let $T = V(H_m) \setminus S$.

Note that by definition, there are no directed edges in $H_m$ from a vertex in $S$ to a vertex in $T$. Indeed, suppose that $e_i = x_i \rightarrow y_i$ with $x_i \not\in T$ and $y_i \in T$. Since $x_i \not\in T$, $x_m \leadsto x_i$ in $H_m$ and so also $x_m \leadsto y_i$.
We get $y_i \not\in T$ which gives us a contradiction (see Figure~\ref{fig:weighting}).
 We construct a series of auxiliary graphs $F_i$ for $i = 0,1,\ldots, m$ on vertex set $T$ as follows. Let $F_0$ be the empty graph on vertex set $T$. For each $1 \le i \le m$, consider the edge $e_i = x_i \rightarrow y_i$ in $H_m$.
\begin{itemize}
    \item If $x_i,y_i \in T$ then let $F_i$ be $F_{i-1}$ with edge $x_i \rightarrow y_i$ added.
    \item If $x_i \in T, y_i \not\in T$, then let $F_i$ be $F_{i-1}$ with loop edge $x_i \rightarrow x_i$ added (with multiplicity if $x_i\rightarrow x_i$ is already included as an edge).
    \item Otherwise, if $x_i \not\in T$, let $F_i = F_{i-1}$.
\end{itemize}
These graphs naturally inherit the edge ordering from $H_m$.

Since $x_m \not\in T$, $e_m = x_m \rightarrow y_m$ is not added to $F_m$ and so we know that $F_m$ has strictly fewer than $m$ edges. We also have that $d_{F_m}(v) = d_{H_m}(v)$ for all $v \in T$, since there are no edges $e_i = x_i \rightarrow y_i$ with $x_i \not\in T$ and $y_i \in T$.  Thus by the inductive hypothesis, the weighting $w_{F_m}:T\rightarrow \N\cup \{0\}$ has $w_{F_m}(v) \ge \delta$ for all $v \in T$.

On the other hand, we can also show inductively that $w_{F_i}(v) = w_{H_i}(v)$ for all $v \in T$ and all $i \le m$. This certainly holds for $i=0$.  Since there are no directed edges from $S$ to $T$ in $H_m$, the set of vertices $v \in T$ with $x_i \leadsto T$ is the same in both $H_i$ and $F_i$, for all $i$. Thus, if $e_i = x_i \rightarrow y_i$ with $x_i \in T$, we can see that for $v \in T$ with $x_i \leadsto v$, we have $w_{H_i}(v) = \max\{w_{H_{i-1}}(v), w_{H_{i-1}}(x_i) + 1 \} = \max\{w_{F_{i-1}}(v), w_{F_{i-1}}(x_i) + 1 \} = w_{F_i}(v)$ and for $v \in T$ with $x_i \not\leadsto v$, we have $w_{H_i}(v) = w_{H_{i-1}}(v) = w_{F_{i-1}}(v) = w_{F_i}(v)$. For the same reason, if $x_i \not\in T$, then adding edge $e_i$ to $H_{i-1}$ to get $H_i$ has no effect on vertex weights within $T$. In particular, we obtain $w_{H_m}(v) = w_{F_m}(v) \ge \delta$ for all $v \in T$ and, consequently, $w_{H_m}(v)\ge\delta$ for all
$v\in V(H_m)$, since we already had that $w_{H_m}(v) \ge \delta$ for $v \in S = V(H_m) \setminus T$. The proof of the lemma is finished.
\end{proof}

Recall that $G_t$ is the semi-random graph after $t$ time-steps. Before we can state our main lemma, we need to introduce a few more definitions. Let $H$ be an oriented graph with a fixed edge order $e_1, e_2, \ldots, e_m$. A homomorphism from $H$ to $G_t$ is a map that respects the edge orientations and edge ordering in the natural way. Formally, a \emph{homomorphism} from  $H$ to $G_t$ is an injective function $\phi: V(H) \rightarrow V(G_t)$ such that:
\begin{enumerate}[a]
    \item if $e = u \rightarrow v$ is a directed edge in $H$ then $\phi(u) \rightarrow \phi(v)$ is a directed edge in $G_t$ which we call $\phi(e)$; and
    \item for $i<j$, the  edge $\phi(e_i)$ was added to $G_t$ at an earlier time-step than the edge $\phi(e_j)$.
\end{enumerate}

For a vertex $v \in V(H)$, define $S(v,H;t)$ to be the set of vertices $u$ in $G_t$ for each of which there is a homomorphism $\phi$ from $V(H)$ to $V(G_t)$ such that $u=\phi(v)$.
Less formally, one can think of $S(v,H;t)$ as being the number of vertices within $G_t$ that look like an image of the vertex $v$ within some copy of the graph $H$.

We also need a notion of the diameter of an oriented graph. For any ordered pair of vertices $u,v \in H$ for which there exists a directed path from $u$ to $v$, let $d(u,v)$ be the length of the shortest such path. Define $diam(H)$ as the maximum value of $d(u,v)$ over all pairs $u,v$ with a directed path from $u$ to $v$. (We use the convention that $diam(H) = 0$ if $H$ is the empty graph.) Note that $diam(H) \le |V(H)| - 1$ and that $diam(H)$ is \emph{not} a monotone function of graphs.

\begin{lemma} \label{lem:weight_related_to_H}
Let $H$ be an oriented graph with a fixed edge order $(e_1,\dots,e_m)$, and let $w_H$ be the vertex weighting defined above. Then, for any strategy $\mathcal{S}$ of the player, any $t < n/2$, and any vertex $v \in V(H)$ with $w = w_H(v)$,
$$
\E \ |S(v,H;t)| \le \frac{t^{w}}{n^{w-1}}\left(2(m!)^{D} - 1 \right)
$$
where $D = \max_i\{diam(H_i)\} \le |V(H)| - 1$.
\end{lemma}


\begin{proof}

We will prove a slightly stronger statement: for any vertex $v \in V(G)$ and any $i \in \{0, 1, \ldots, m\}$,
\begin{equation}\label{eq:induction}
    \E \ |S(v,H_i;t)| \le \frac{t^{w_i(v)}}{n^{w_i(v)-1}}\left(2(w_i(v)!)^{D} - 1 \right),
\end{equation}
where we use the shorthand $w_i = w_{H_i}$. Then, the lemma will follow by taking $i=m$ and observing that $w_i(v)!\le m!$. 
The proof is by induction on $i$.

The base case is trivial. Indeed, if $i = 0$, then by the definition of $w_{0}$, the weighting is identically zero.  Clearly, for all $v \in V(H)$,
$$
S(v,H_{0};t) = n = \frac{t^0}{n^{-1}}\left(2(0!)^{D} - 1\right),
$$
and so the desired inequality~(\ref{eq:induction}) holds.

For the inductive step, suppose that $i \ge 1$ and that~(\ref{eq:induction}) holds for $H_{i-1}$. We will show that it also holds for $H_i$. If $w_i(v) = 0$, then clearly $|S(v,H_i;t)| \le n = \frac{t^0}{n^{-1}}\left(2(0!)^{D} - 1\right)$ and we are done. If $w_i(v) = 1$, then $v$ must be in some edge in $H_i$. If $v$ has a positive out-degree in $H_i$, then each vertex in $S(v,H_i;t)$ has a positive out-degree in $G_t$. Otherwise, $v$ has a positive in-degree in $H_i$, and then each vertex in $S(v,H_i;t)$ has a positive in-degree in $G_t$. Either way, $|S(v,H_i;t)| \le t = \frac{t^1}{n^{0}}\left(2(1!)^{D} - 1\right)$ and we are done again.

Moreover, as, obviously,
$$
|S(v,H_i;t)| \le |S(v,H_{i-1};t)|,
$$
the result follows immediately for all vertices $v$ with $w_i(v) = w_{i-1}(v)$. Consequently, we only need to consider vertices $v$ where $w_i(v) \ge 2$ and $w_i(v) \ne w_{i-1}(v)$.

Let  $e_i=x \rightarrow y$  and set $w_i(x)=w$. The condition $w_i(v) \ne w_{i-1}(v)$ only holds if  there is a directed path from $x$ to $v$ in $H_i$, in which case $w_i(v) = w$ too. We will show by induction on the distance $d_i(x,v)$ from $x$ to $v$  in $H_i$ that
\begin{equation}\label{eqn:recursion_on_d(x,v)}
    \E \ |S(v,H_i;t)| \le \frac{t^{w}}{n^{w-1}}\left(2w^{d_i(x,v)}\left((w-1)!\right)^{D} - 1 \right).
\end{equation}
Since $d_i(x,v) \le diam(H_i) \le D$, this will suffice to prove~(\ref{eq:induction}) and so to finish the proof of the lemma.

From now on we suppress the subscript $i$ in $d_i(x,y)$.
First consider the case $d(x,v)=0$, that is, $v=x$. If  $u\in S(x,H_i;t)$, then there must be some time $t' < t$ so that  $u \in S(x, H_{i-1};t')$ and $u$ was selected by the semi-random process at time $t'+1$, when the image of the edge $e_i$ was added to create a copy of $H_i$. The probability that some vertex in $S(x, H_{i-1};t')$ was selected by the semi-random process at time $t'+1$ is  $|S(x,H_{i-1};t')|/n$.
Thus,
$$\E \ |S(x,H_i;t)| \le \sum_{t'=0}^{t-1} \frac{|S(x,H_{i-1};t')|}{n}$$
and, by the linearity of expectation and \eqref{eq:induction}, valid for $i-1$, we have
\begin{align*}
\E \ |S(x,H_i;t)| &\le \sum_{t'=0}^{t-1} \frac{\E \ |S(x,H_{i-1};t')|}{n} \\
&\le \sum_{t'=0}^{t-1} \frac{(t')^{w-1}}{n\left(n^{w-2}\right)}\left(2\left((w-1)!\right)^{D} - 1 \right) \\
&\le \frac{t^{w}}{n^{w-1}}\left(2\left((w-1)!\right)^{D} - 1 \right),
\end{align*}
as required.

Now, consider the case $d(x,v)>0$, that is,  $v \ne x$ and there is a directed path from $x$ to $v$, and suppose that the hypothesis~(\ref{eqn:recursion_on_d(x,v)}) holds for all $u$ with $d(x,u) < d(x,v)$. We fix a directed path from $x$ to $v$ of minimum length. Let $u$ be the vertex preceding $v$ on this path, so  $d(x,u) = d(x,v) - 1$. Observe that, by the definition of the weight function,  $w_i(u)= w$.

The number of vertices in $S(v,H_i;t)$ is bounded by the number of edges in $G_t$ that are the images of the edge $u \rightarrow v$ under some homomorphism $\phi$ from $H_i$ to $G_t$. We partition the vertices in $S(u,H_i;t)$ into classes according to how many of the edges they are incident to are the images of $u \rightarrow v$ under some homomorphism. If a vertex in $ S(u,H_i;t)$ is incident to exactly $a$ such edges, it contributes at most $a$ vertices to $S(v, H_i;t)$.

Thus, the total contribution to $S(v,H_i;t)$ from all vertices in $ S(u,H_i;t)$ that are incident to at most $w$ such edges is at most $w |S(u,H_i;t)|$.
 On the other hand, the expected number of vertices in $S(u,H_i;t)$ that are incident to exactly $a > w$ edges that are images of $u \rightarrow v$ is, by Lemma \ref{lem:pick_same_vertex}(a), at most $\frac{2t^a}{a!n^{a-1}}$. Combining these estimates, we have
\begin{align*}
    \E \ |S(v,H_i;t)| &\le w \E \ |S(u,H_i;t)| + 2\sum_{a\ge w+1} a\frac{t^a}{a!n^{a-1}}\\
    &\le w \E \ |S(u,H_i;t)|  + \frac{2t^w}{w!n^{w-1}}\left( \frac{t}{n} + \frac{t^2}{n^2} + \ldots \right) \\
    &\le w \frac{t^{w}}{n^{w-1}}\left(2w^{d(x,u)}\left((w-1)!\right)^{D} - 1 \right) + \frac{t^w}{n^{w-1}} \cdot \frac{t}{n-t},
\end{align*}
since $w \ge 2$.  It follows that
\begin{align*}
    \E \ |S(v,H_i;t)| &\le \frac{t^w}{n^{w-1}} \left(2w^{d(x,u)+1}\left((w-1)!\right)^{D} - w +   \frac{t}{n-t} \right) \\
    &\le  \frac{t^w}{n^{w-1}} \left(2w^{d(x,v)}\left((w-1)!\right)^{D} - 1 \right),
\end{align*}
as $t < n/2$. Thus, inequality~(\ref{eqn:recursion_on_d(x,v)}) holds for all $v$ with a directed path from $x$ to $v$. This finishes the proof of the lemma.
\end{proof}

Now we combine the two lemmas to prove the theorem.

\begin{proof}[Proof of Theorem~\ref{thm:lower_bound}]
Let $H$ be a graph on $k$ vertices and $m$ edges with degeneracy $d$. Let $H'$ be the (non-empty) $d$-core of $H$ so, in particular, $H'$ has minimum degree at least $d$. We will show that, regardless of the strategy used by the player, a.a.s.\ $H'$ is not a subgraph of $G_t$ as long as $t = o(n^{(d-1)/d})$. As a result, the same is true for $H \supseteq H'$.

As mentioned at the beginning of this section, it is enough to show that the player cannot create a copy of $H$ following a specific (but arbitrarily chosen) orientation and order of the edges of $H$. Clearly, there are $2^m m!$ different configurations to select from (which is a large constant, but it does not depend on $n$). We may consider $2^m m!$ auxiliary games, one for each configuration, on top of the regular game. Each game (both the auxiliary ones and the original one) is played by $2^m m!+1$ perfect players aiming to achieve their own respective goals. All the games are coupled in a natural way, that is, exactly the same squares are presented by the semi-random process to each of the players.

Fix an edge ordering and an orientation of the edges of $H$, and consider a perfect player playing the corresponding auxiliary game.
Applying Lemma~\ref{lem:weight_>_min_deg} to the $d$-core $H'$ of $H$, we see that  $w_{H'}(v) \ge d$ for each vertex $v \in V(H')$.
Thus, by Lemma~\ref{lem:weight_related_to_H}, for each vertex $v \in V(H')$
$$
\E \ |S(v,H';t)| \le \frac{t^{d}}{n^{d-1}}\left(2(e(H')!)^{|V(H')|-1} - 1 \right) = O \left( \frac{t^{d}}{n^{d-1}} \right).
$$

Note that the above bound holds for \emph{all} vertices $v \in V(H')$. This property is slightly stronger than we need as we only need it for \emph{one} vertex of $V(H')$.
Let us fix then an arbitrary vertex $v_0$ of $H'$. 
If $t = o\left(n^{(d-1)/d}\right)$, then we have $\E \ |S(v_0,H';t)| = o(1)$ and so, by the first moment method, a.a.s.\ there are no vertices in $S(v_0,H';t)$.
It follows that a.a.s.\ there is no copy of $H'$, and thus of $H$, in $G_t$ with this given order and orientation of edges. In other words, the player playing this specific auxiliary game does not win the game a.a.s at time $t = o\left(n^{(d-1)/d}\right)$.

This holds for every edge ordering and every orientation of the edges of $H$. As mentioned earlier, the number of such orderings and orientations is a constant depending only on $m=|E(H)|$. Thus, by the union bounds over all auxiliary games, a.a.s.\ all players playing auxiliary games lose their own respective games. It follows that a.a.s.\ a perfect player playing the original game loses too (if not, the other players could all mimic the same strategy, and one of them would win her game). The proof is finished.
\end{proof}

\section{Proofs for hypergraphs when $r=1$}\label{sec:hypergraphs r=1}

In the case when $r=1$, for each step $t$ of the semi-random process for hypergraphs, a single vertex $u_t$ is chosen uniformly at random from $[n]$, the same as for the process on graphs. The player then replies by selecting a set of $s-1$ vertices $V_t$, and the edge $\{u_t\} \cup V_t$ is added.

The proofs of Theorems \ref{thm:hypergraph_r=1_upper_bound} and \ref{thm:hypergraph_r=1_lower_bound} follow the same approach as in the graph case considered in Section~\ref{sec:graphs} and, therefore, we only sketch them here emphasizing the required differences.
The proof of Theorems \ref{thm:hypergraph_r=1_upper_bound} is again based on Lemma \ref{lem:pick_same_vertex} and, indeed, proceeds mutatis mutandis.

 \begin{proof}[Proof of Theorem~\ref{thm:hypergraph_r=1_upper_bound}]
Follow the same approach as in the proof of Theorem \ref{thm:upper_bound}, dividing the process into phases where in each phase, the next vertex according to the degeneracy ordering is added. In each phase, we must create all edges in which the new vertex is the last vertex. Since $r=1$ these can be constructed in exactly the same way.
\end{proof}

Now we outline the proof of Theorem \ref{thm:hypergraph_r=1_lower_bound} by discussing the necessary changes in the proof of the graph counterpart, which is Theorem \ref{thm:lower_bound}.
Let $H$ be a hypergraph on $k$ vertices and $m$ edges. For the purposes of the proof of Lemma~\ref{lem:hypergraph_weight_>_min_deg}, we allow $H$ to be not necessarily uniform, but with every edge containing at most $s$ vertices, and we also allow there to be potentially multiple copies of edges on $<s$ vertices. Call such a hypergraph \emph{$s$-bounded}. This is analogous to the so-called loopy graphs used in the proof of Lemma \ref{lem:weight_>_min_deg}.

Fix an ordering of the edges $(e_1,e_2,\ldots,e_m)$ of $H$, and for each edge $e_i$ fix a leading vertex $x_i$.
Given such an $H$, define an auxiliary directed $graph(H)$ on the same vertex set $V$ where for each edge $e_i \in H$, we have that $graph(H)$ contains the directed edges $x_i \rightarrow u$ for every $u \in e_i \setminus \{x_i\}$. Note that the only property of $graph(H)$ we will use is whether two vertices have a directed path between them, so while we will add any arcs in opposite directions, we do not add any parallel arcs that are in the same direction.

As in the graph case, we recursively define a weight function $w_H:V \rightarrow \N\cup\{0\}$ on the vertices of $H =(V,E)$ that is dependent on the edge order and choice of leading vertices. Let $H_0$ be the empty hypergraph on vertex set $V$ and define the weighting $w_{H_0}:V \rightarrow \N\cup\{0\}$ to be uniformly zero.
For $1 \le i \le m = |E|$, let $H_i$ have vertex set $V$ and edge set $\{ e_1, e_2, \ldots,e_i \}$ (so $H_m = H$). In particular, $H_i$ is $H_{i-1}$ with edge $e_i$ added. Define $w_{H_i}:V \rightarrow \N\cup\{0\}$ by
$$
    w_{H_i}(x_i) = w_{H_{i-1}}(x_i) + 1
$$
and for all other vertices $v$,
$$
    w_{H_i}(v) = \begin{cases} \max\{w_{H_i}(x_i), w_{H_{i-1}}(v)\} &\text{if $x_i \leadsto v$ in $graph(H_i)$} \\
     w_{H_{i-1}}(v) &\text{otherwise,}
     \end{cases}
$$
where $x_i \leadsto v$ denotes that there is a directed path from $x_i$ to $v$.

\begin{lemma}\label{lem:hypergraph_weight_>_min_deg}
If an $s$-bounded hypergraph $H$ has minimum degree $\delta = \delta(H)$, then $w_H(v) \ge \delta$ for every vertex $v$.
\end{lemma}
\begin{proof}
The proof follows the same approach as the proof of Lemma \ref{lem:weight_>_min_deg}, using induction on the number of edges $m = |E(H)|$.
 The base case  $m=0$ is trivial.

For the inductive step, let $\delta=\delta(G_m)$. If $G_{m-1}$ has minimum degree $\delta$ then we are done, so we assume $G_{m-1}$ has minimum degree $\delta - 1$ and that $w_{G_{m-1}}(v) \ge \delta - 1$ for all $v$. We have $w_{G_m}(x_m) = w_{G_{m-1}}(x_m) + 1 \ge \delta$. Let $S$ be the set of all vertices $v$ where there is a directed path from $x_m$ to $v$ in $graph(G_m)$.

We know that $x_m \in S$ and $w_{G_m}(v) \ge w_{G_m}(x_m) \ge \delta$ for any $v \in S$. If $S = V(G_m)$ then we are done, so suppose not and let $T = V(G_m) \setminus S$. 
 We construct a series of auxiliary hypergraphs $F_i$ for $i = 0,1,\ldots, m$ on vertex set $T$ as follows. Let $F_0$ be the empty graph on vertex set $T$. For each $1 \le i \le m$, consider the edge $e_i$ in $G_m$.
 \begin{itemize}
     \item If $x_i \in T$ then let $F_i$ be $F_{i-1}$ with edge $e_i \cap T$ added (with multiplicity if already included).
     \item Otherwise, if $x_i \not\in T$, let $F_i = F_{i-1}$.
 \end{itemize}

$F_m$ has strictly fewer than $m$ edges and for all $v \in T$, the degree $d_{F_m}(v) = d_{G_m}(v)$. Thus by the inductive hypothesis, the weighting $w_{F_m}:T\rightarrow \N\cup \{0\}$ has $w_{F_m}(v) \ge \delta$ for all $v \in T$.
 On the other hand, we can also show that $w_{F_i}(v) = w_{G_i}(v)$ for all $v \in T$ and all $i \le m$ by the same argument as in Lemma~\ref{lem:weight_>_min_deg}.
\end{proof}

Let $G_t = G_t^{(1,s)}$ be the semi-random hypergraph after $t$ time-steps and let $u_t$ be the single vertex in the randomly chosen set $U_t$ at time $t$.
Before we can state our main lemma, we need to generalize some of our earlier definitions to hypergraphs. Let $H$ be an hypergraph with a fixed edge order $e_1, e_2, \ldots, e_m$ where each edge $e_i$ is assigned a leading vertex $x_i \in e_i$. A homomorphism from $H$ to $G_t$ is a map that respects the leading vertices and edge ordering in the natural way (analogously to the graph case). For a vertex $v \in V(H)$, define $S(v,H;t)$ to be the set of vertices in $G_t$ which are the image $\phi(v)$ for some homomorphism $\phi$ from $V(H)$ to $V(G_t)$.
We also define the diameter $diam(H)$ of the hypergraph to be the diameter of the auxiliary graph $graph(H)$, in the same sense as defined in Section~\ref{sec:graphs}.

\begin{lemma} \label{lem:hypergraph_weight_related_to_H}
Let $H$ be an $s$-graph with a fixed edge order $e_1, e_2, \ldots, e_m$ where each edge $e_i$ is assigned a leading vertex $x_i \in e_i$. Let $w_H$ be the vertex weighting defined above. Then, for any strategy $\mathcal{S}$ of the player, any $t < n/2$, and any vertex $v \in V(H)$ with $w = w_H(v)$,
$$
\E \ |S(v,H;t)| \le \frac{t^{w}}{n^{w-1}}\left(2\left((s-1)^{m}m!\right)^{D} - 1 \right)
$$
where $D = \max_i\{diam(H_i)\} \le |V(H)| - 1$.
\end{lemma}

\begin{proof}
The proof of this lemma follows the exact same approach as the proof of Lemma \ref{lem:weight_related_to_H}.
Specifically, one can show by induction on $i$ that if hypergraph $H$ has $m$ edges then for any vertex $v \in V(H)$ and any $i \in \{0, 1, \ldots, m\}$,
\begin{equation*}\label{eq:hypergraph_induction}
    \E \ |S(v,H_i;t)| \le \frac{t^{w_i}}{n^{w_i-1}}\left(2\left((s-1)^{w_i}w_i!\right)^{D} - 1 \right),
\end{equation*}
where $w_i = w_{H_i}(v)$. There are two changes needed in the proof. The first is to use the directed paths given by the auxiliary graphs $graph(H_i)$, whereas in the original proof the $H_i$'s were themselves oriented.

The second is where the extra $(s-1)^{wD}$ factor arises. It comes when considering the case where $d(x,v)>0$  and there is a directed path from $x$ to $v$. We fix a directed path from $x$ to $v$ in $graph(H_i)$ of minimum length and let $u$ be the vertex preceding $v$ on this path, so  $d(x,u) = d(x,v) - 1$.

There must be a hyperedge $e$ of $H_i$  containing both $u$ and $v$ in which $u$ is the leading vertex.
We partition the vertices in $S(u,H_i;t)$ into classes according to how many of the hyperedges they are incident to are the images of $e$ under some homomorphism. If a vertex in $ S(u,H_i;t)$ is incident to exactly $a$ such hyperedges, then it contributes at most $(s-1)a$ vertices to $S(v,H_i;t)$. This is the number of hyperedges multiplied by the number of other vertices in each hyperedge.

Thus, the total contribution to $S(v,H_i;t)$ from all vertices in $ S(u,H_i;t)$ that are incident to at most $w_i$ such edges is at most $(s-1)w_i |S(u,H_i;t)|$.
 On the other hand, the expected number of vertices in $S(u,H_i;t)$ that are incident to exactly $a > w_i$ edges that are images of $e$ is, by Lemma \ref{lem:pick_same_vertex}(a), at most $\frac{2t^a}{a!n^{a-1}}$. Combining these estimates, we have
\begin{align*}
    \E \ |S(v,H_i;t)| &\le (s-1)w_i \E \ |S(u,H_i;t)| + 2\sum_{a\ge w_i+1} (s-1)a\frac{t^a}{a!n^{a-1}}\\
\end{align*}
and the rest of the proof follows along the same lines as before.
\end{proof}

Now we combine the two lemmas to prove the lower bound.

\begin{proof}[Proof of Theorem~\ref{thm:hypergraph_r=1_lower_bound}]
Let $H$ be a hypergraph on $k$ vertices and $m$ edges with degeneracy $d$. Let $H'$ be the (non-empty) $d$-core of $H$ so, in particular, $H'$ has minimum degree at least $d$.

Using the same coupling argument as in the proof of
Theorem  \ref{thm:lower_bound} one can see that, regardless of the strategy used by the player, a.a.s.\ $H'$ is not a sub-hypergraph of $G_t$ as long as $t = o(n^{(d-1)/d})$. As a result, the same is true for $H \supseteq H'$.
\end{proof}

\section{Open Problems}

Let us finish the paper with a few open problems. The value of $\tau_{\mathcal{P}_H}$ is determined for \emph{any} uniform hypergraph $H$  (if $r=1$); see Theorem~\ref{thm:any_hypergraph}. (This covers the case when $H$ is a graph.) In fact, we establish that a.a.s.\ one may construct $H$ in $t \gg n^{(d-1)/d}$ rounds but cannot do it in $t = o\left(n^{(d-1)/d}\right)$ rounds, where 
$d$ is the degeneracy of~$H$.

Note that the nature of our proofs means that they work equally well when $H$ has parallel edges and/or if $H$ is a non-uniform hypergraph with all edges of size at most $s$. (In the graph case, this amounts to a multi-graph with loops.) The definition of degeneracy has to be adjusted accordingly but the proofs go through without any alteration.

It remains to investigate the probability of success after $t = c n^{(d-1)/d}$ rounds, where $c$ is some fixed positive constant. It is natural to expect that an optimal strategy produces $(1+o(1))f(c)$ copies of $H$ in expectation for some deterministic function $f(c)$, and then the limiting probability that the strategy fails falls into the open interval $(0,1)$. Under some structural properties of $H$, it may actually tend to $\exp(-f(c))$, per analogy with the purely random (hyper)graph (see~\cite[Chapter 3]{JLR}). However, determining an optimal strategy and analyzing it might be challenging.

\begin{problem}
Determine the limiting probability that $\mathcal{P}_H$ holds for $t=c n^{(d-1)/d}$ for some positive constant $c$.
\end{problem}

Note also that Theorem~\ref{thm:any_hypergraph} applies to a fixed hypergraph $H$. If the order of $H$ is an increasing function of $n$, then our results do \emph{not} apply. In the extreme but quite natural case, $H$ may have $n$ vertices, so the player is after a spanning sub-hypergraph of $G^{(r,s)}_t$ isomorphic to $H$. For graphs, as mentioned in the introduction, we know that a.a.s.\ one may construct a copy of a graph with bounded degree in $(3\Delta/2+o(\Delta))n$ rounds~\cite{beneliezer2020fast}. However, these upper bounds are asymptotic in $\Delta$. When $\Delta$ is constant in $n$, such as in both the perfect matching and the Hamiltonian cycle setting, determining the optimal dependence on $\Delta$ of the number of rounds needed to construct $H$ remains open. A good starting point (apart from matchings and Hamiltonian cycles already considered in~\cite{gao2020hamilton,gao2022fully,frieze2022hamilton}) might be to investigate $F$-factors, that is, spanning subgraphs whose all components are isomorphic to a fixed connected graph.

\begin{problem}
Given a graph  $F$, estimate the number $t$ of rounds needed to a.a.s.\ construct an $F$-factor in $G_t$ on $n$ vertices, where $n$ is divisible by $|V(F)|$.
\end{problem}

For hypergraphs, we know much less in the case when $r$, the number of randomly selected vertices at each step, is greater than $1$.
We can define $\tau^{(r)}_{\mathcal{P}_H}$ analogously to $\tau_{\mathcal{P}_H}$ by  replacing $G^{(s)}_t = G^{(1,s)}_t$ with  $G^{(r,s)}_t$. In particular, $\tau^{(1)}_{\mathcal{P}_H}$ is the $\tau_{\mathcal{P}_H}$ explored in detail in this paper.
The most ambitious goal would be to obtain a general formula for $\tau^{(r)}_{\mathcal{P}_H}$.
\begin{problem}
Given an $s$-graph $H$ and an integer $2\le r<s$, determine $\tau^{(r)}_{\mathcal{P}_H}$.
\end{problem}
In~\cite{behague2024}, we obtained a general lower bound $\tau^{(r)}_{\mathcal{P}_H}\ge n^{r-(k-s+r)/m}$, where $k=|V(H)|$ and $m=|E(H)|$, showed its optimality for certain classes of hypergraphs and better bounds for some others. However, the general question remains wide open.
To answer it, we believe, one would need to come up with entirely new strategies of the player.



\bibliographystyle{abbrv}
\bibliography{refs}

\end{document}